\documentclass{amsproc}
\usepackage{amssymb}
\usepackage{graphicx}
\usepackage{amscd}
\usepackage{amsmath}
\theoremstyle{plain}

\newtheorem{corollary}{Corollary}

\newtheorem{definition}{Definition}

\newtheorem{lemma}{Lemma}

\newtheorem{agreement}{Agreement}

\newtheorem{proposition}{Proposition}
\newtheorem{remark}{Remark}

\newtheorem{theorem}{Theorem}

\numberwithin{equation}{section}

\begin{document}

\title[Dynamic inverse problem and classical moment problems]{Dynamic inverse problem for special system
associated with Jacobi matrices and classical moment problems.}

\author{Alexander Mikhaylov} 
\address{St. Petersburg   Department   of   V.A. Steklov    Institute   of   Mathematics
of   the   Russian   Academy   of   Sciences, 7, Fontanka, 191023
St. Petersburg, Russia and Saint Petersburg State University,
St.Petersburg State University, 7/9 Universitetskaya nab., St.
Petersburg, 199034 Russia.} \email{mikhaylov@pdmi.ras.ru}

\author{Victor Mikhaylov} 
\address{St.Petersburg   Department   of   V.A.Steklov    Institute   of   Mathematics
of   the   Russian   Academy   of   Sciences, 7, Fontanka, 191023
St. Petersburg, Russia and Saint Petersburg State University,
St.Petersburg State University, 7/9 Universitetskaya nab., St.
Petersburg, 199034 Russia.} \email{ftvsm78@gmail.com}

\keywords{problem of moments, Boundary control method, Jacobi
matrices}
\date{March, 2019}

\maketitle

\noindent {\bf Abstract.} We consider the Hamburger, Stieltjes and
Hausdorff moment problems, that are problems of the construction
of a Borel measure supported on a real line, on a half-line or on
an interval $(0,1)$, from a prescribed set of moments. We propose
a unified approach to these three problems based on using the
auxiliary dynamical system with the discrete time associated with
a semi-infinite Jacobi matrix. It is show that the set of moments
determines the inverse dynamic data for such a system. Using the
ideas of the Boundary Control method for every $N\in \mathbb{N}$
we can recover the spectral measure of a $N\times N$ block of
Jacobi matrix, which is a solution to a truncated moment problem.
This problem is reduced to the finite-dimensional generalized
spectral problem, whose matrices are constructed from moments and
are connected with well-known Hankel matrices by simple formulas.
Thus the results on existence of solutions to Hamburger, Stieltjes
and Hausdorff moment problems are naturally given in terms of
these matrices. We also obtain results on uniqueness of the
solution of moment problems, where as a main tool we use the
Krein-type equations of inverse problem.

\section{Introduction.}

Given a sequence of numbers $s_0,s_1,s_2,\ldots$ which are called
moments, the classical moment problem consists in finding a Borel
measure $\rho$ such that
\begin{equation}
\label{Moment_eq} s_k=\int_{-\infty}^\infty
\lambda^k\,d\rho(\lambda),\quad k=0,1,2,\ldots.
\end{equation}
When $\operatorname{supp}{\rho}=\mathbb{R}$ the problem is called
Hamburger moment problem, when
$\operatorname{supp}{\rho}=(0,+\infty)$ the problem is called the
Stieltjes moment problem, and when
$\operatorname{supp}{\rho}=(0,1)$ the problem is called Hausforff
moment problem. These three problems have received a lot of
attention in the lust century, to mention \cite{A,T,H1,H2,S} and
references therein. In the present paper we offered an unified
approach to these three classical problems bases on considering an
auxiliary dynamical system with discrete time for Jacobi matrix
\cite{MM,MM3,MMS} and ideas of the Boundary Control (BC) method
\cite{B07,B17} of solving the inverse dynamic problems for
hyperbolic dynamical systems.

In the second section we consider initial-boundary value problems
for dynamical systems with discrete time associated with
semi-infinite and finite Jacobi matrices. Following
\cite{MM,MM2,MM3} we derive a dynamic and spectral representations
of their systems. We introduce the operators of the BC method and
show that \emph{the response operator}, i.e the discrete analog of
a dynamic Dirichlet-to-Neumann map for these systems (operators of
this type are used as inverse data in dynamic inverse problems
\cite{B07,B17}) has a form of convolution. The kernel of the
response operator, which is called \emph{response vector}, admits
a spectral representation in terms of a spectral measure of
corresponding Jacobi matrix. This fact establishes the
relationship between spectral (measure) and dynamic (response
vector) data and gives a possibility to apply some ideas of the BC
method \cite{B2001,AMM} to solving the truncated moment problem.

In the third section we solve the truncated moment problem by
extracting spectral data (i.e. the spectral measure of $N\times N$
block of Jacobi matrix) from the response vector. The main results
are given in Theorems \ref{teor} and \ref{Propos_EL}, which say
that the solution to a truncated moment problem can be constructed
by solving special finite dimensional generalized spectral
problem, in which the matrices are connected with classical Hankel
matrices (see \cite{Ahiez,S}) constructed from moments by simple
transformation. Then the results on the existence of solution to
all three moment problems are given in terms of inequalities for
these matrices. Note that classical results for Hausdorff moment
problem \cite{H1,H2} are given in completely different terms.

In the last section we obtain the results on uniqueness of the
solution to Hamburger, Stieltjes and Hausdorff moment problems.
The main tools in our considerations are classical Weyl-type
results on the index of Jacobi matrix \cite{Ahiez,S} and Krein
equations of inverse problem in dynamic form. For continuous
systems such equations were derived firstly in \cite{BL} and in
the framework of the BC method in \cite{AM,BM}; for the discrete
systems they were derived in \cite{MM,MM2,MM3}. We also compare
the results on existence for Hausdorff moment problem obtained in
the paper with classical results of Hausdorff \cite{H1,H2,T}.

\section{Dynamical systems with discrete time associated with Jacobi matrix. Operators of the BC method.}

In this section we outline some results on forward problems for
dynamical system with discrete time associated with finite and
semi-infinite Jacobi matrices obtained in \cite{MM,MM2,MM3}.

\subsection{Finite Jacobi matrices}

For a given sequence of positive numbers $\{a_0,$ $a_1,\ldots\}$
(in what follows we assume $a_0=1$) and real numbers $\{b_1,
b_2,\ldots \}$, we denote by $A$ a semi-infinite Jacobi matrix
\begin{equation}
\label{Jac_matr}
A=\begin{pmatrix} b_1 & a_1 & 0 & 0 & 0 &\ldots \\
a_1 & b_2 & a_2 & 0 & 0 &\ldots \\
0 & a_2 & b_3 & a_3 & 0 & \ldots \\
\ldots &\ldots  &\ldots &\ldots & \ldots &\ldots
\end{pmatrix}.
\end{equation}
For $N\in \mathbb{N}$, by $A^N$ we denote the $N\times N$ Jacobi
matrix which is a block of (\ref{Jac_matr}) consisting of the
intersection of first $N$ columns with first $N$ rows of $A$.

Introduce the notation $\mathbb{N}_0=\mathbb{N}\cup\{0\}$, and
consider a dynamical system with discrete time associated with a
finite Jacobi matrix $A^N$:
\begin{equation}
\label{Jacobi_dyn_int} \left\{
\begin{array}l
v_{n,t+1}+v_{n,t-1}-a_nv_{n+1,t}-a_{n-1}v_{n-1,t}-b_nv_{n,t}=0,\,\, t\in \mathbb{N}_0,\,\, n\in 1,\ldots, N,\\
v_{n,\,-1}=v_{n,\,0}=0,\quad n=1,2,\ldots,N+1, \\
v_{0,\,t}=f_t,\quad v_{N+1,\,t}=0,\quad t\in \mathbb{N}_0,
\end{array}\right.
\end{equation}
by an analogy with continuous problems \cite{B07,AM,BM}, we treat
the real sequence $f=(f_0,f_1,\ldots)$ as a \emph{boundary
control}. Fixing a positive integer $T$ we denote by
$\mathcal{F}^T$ the \emph{outer space} of the system
(\ref{Jacobi_dyn_int}): $\mathcal{F}^T:=\mathbb{R}^T$, $f\in
\mathcal{F}^T$, $f=(f_0,\ldots,f_{T-1})$, we use the notation
$\mathcal{F}^\infty=\mathbb{R}^\infty$ when control acts for all
$t\geqslant 1$. The solution to (\ref{Jacobi_dyn_int}) is denoted
by $v^f$. Note that (\ref{Jacobi_dyn_int}) is a discrete analog of
an initial boundary value problem for a wave equation with a
potential on an interval with the Dirichlet control at the left
end and Dirichlet condition at the right end. The operator
corresponding to a finite Jacobi matrix $A^N$ and Dirichlet
condition at $n=N+1$ we also denote by $A^N:\mathbb{R}^N\mapsto
\mathbb{R}^N$:
\begin{equation*}
(A^N\psi)_n=\begin{cases}b_1\psi_1+a_1\psi_2,\quad n=1,\\
a_{n}\psi_{n+1}+a_{n-1}\psi_{n-1}+b_n\psi_n,\quad
2\leqslant n\leqslant N-1,\\
a_{N-1}\psi_{N-1}+b_N\psi_N,\quad n=N,
\end{cases}
\end{equation*}

Denote by $\phi_n(\lambda)$ a solution to the Cauchy problem for
the following difference equation
\begin{equation}
\label{Phi_def}
\begin{cases} a_n\phi_{n+1}+a_{n-1}\phi_{n-1}+b_n\phi_n=\lambda\phi_n,\quad n\geqslant 1,\\
\phi_0=0,\,\,\phi_1=1.
\end{cases}
\end{equation}
Thus $\phi_k(\lambda)$ is a polynomial of degree $k-1$. Denote by
$\{\lambda_k\}_{k=1}^N$ the roots of the equation
$\phi_{N+1}(\lambda)=0$, it is known \cite{Ahiez,S} that they are
real and distinct. We introduce the vectors $\phi^n\in
\mathbb{R}^N$ by the rule $\phi^n_i:=\phi_i(\lambda_n)$,
$n,i=1,\ldots,N,$ and define the numbers $\rho_k$ by
\begin{equation*}
(\phi^k,\phi^l)=\delta_{kl}\rho_k,\quad
k,l=1,\ldots,N,
\end{equation*}
where $(\cdot,\cdot)$ is a scalar product in $\mathbb{R}^N$.
\begin{definition}
The set of pairs
\begin{equation*}
\{\lambda_k,\rho_k\}_{k=1}^N
\end{equation*}
is called Dirichlet spectral data of operator $A^N$.
\end{definition}
\begin{definition}
For $f,g\in \mathcal{F}^\infty$ we define the convolution
$c=f*g\in \mathcal{F}^\infty$ by the formula
\begin{equation*}
c_t=\sum_{s=0}^{t}f_sg_{t-s},\quad t\in \mathbb{N}\cup \{0\}.
\end{equation*}
\end{definition}
Denote by $\mathcal{T}_k(2\lambda)$ the Chebyshev polynomials of
the second kind: they are obtained as a solution to the following
Cauchy problem:
\begin{equation}
\label{Chebysh} \left\{
\begin{array}l
\mathcal{T}_{t+1}+\mathcal{T}_{t-1}-\lambda \mathcal{T}_{t}=0,\\
\mathcal{T}_{0}=0,\,\, \mathcal{T}_1=1.
\end{array}
\right.
\end{equation}
In \cite{MM,MM3} the following expression for the solution $v^f$
was proved:
\begin{proposition}
The solution to (\ref{Jacobi_dyn_int}) admits a representation
\begin{equation}
\label{Jac_sol_rep_int} v^f_{n,t}=
\begin{cases}
\sum_{k=1}^N c_t^k\phi^k_n,\quad n=1,\ldots,N,\\
f_t,\quad n=0.
\end{cases},\quad
c^k=\frac{1}{\rho_k}\mathcal{T}\left(\lambda_k\right)*f.
\end{equation}
\end{proposition}
The \emph{inner space} of dynamical system (\ref{Jacobi_dyn_int})
is denoted by $\mathcal{H}^N:=\mathbb{R}^N$, $h\in \mathcal{H}^T$,
$h=(h_1,\ldots, h_N)$. By (\ref{Jac_sol_rep_int}) we have that
$v^f_{\cdot,\,T}\in \mathcal{H}^N$. For the system
(\ref{Jacobi_dyn_int}) the \emph{control operator}
$W^T_{N}:\mathcal{F}^T\mapsto \mathcal{H}^N$ is defined by the
rule
\begin{equation*}
W^T_{N}f:=v^f_{n,\,T},\quad n=1,\ldots,N.
\end{equation*}

The input $\longmapsto$ output correspondence in the system
(\ref{Jacobi_dyn_int}) is realized by a \emph{response operator}:
$R^T:\mathcal{F}^T\mapsto \mathbb{R}^T$, defined by the formula
\begin{equation} \label{R_def_int}
\left(R^T_Nf\right)_t=v^f_{1,\,t}, \quad t=1,\ldots,T.
\end{equation}
This operator has a form of a convolution:
\begin{equation*}
\left(R^T_Nf\right)=r^N*f_{\cdot-1},
\end{equation*}
where the convolution kernel is called a \emph{response vector}:
$r^N=(r^N_0,r^N_1,\ldots,r^N_{T-1})$. The response operator plays
the role of dynamic inverse data \cite{B07,B17}, for the discrete
systems see \cite{MM,MM3}.

The \emph{connecting operator} $C^T_{N}: \mathcal{F}^T\mapsto
\mathcal{F}^T$ for the system (\ref{Jacobi_dyn_int}) is defined
via the quadratic form: for arbitrary $f,g\in \mathcal{F}^T$ one
has that
\begin{equation*}
\left(C^T_{N} f,g\right)_{\mathcal{F}^T}=\left(v^f_{\cdot,\,T},
v^g_{\cdot,\,T}\right)_{\mathcal{H}^N}=\left(W^T_{N}f,W^T_{N}g\right)_{\mathcal{H}^N},\quad
C^T_N=\left(W^T_{N}\right)^*W^T_{N}.
\end{equation*}

The speed of a wave propagation in the system
(\ref{Jacobi_dyn_int}) is finite, which implies the following
dependence of inverse data on coefficients $\{a_n,b_n\}$: for
$M\in \mathbb{N}$, $M\leqslant N,$ the element $v^f_{1,2M-1}$
depends on $\left\{a_1,\ldots,a_{M-1}\right\},$
$\left\{b_1,\ldots,b_{M}\right\}$, on observing this we can
formulate the following
\begin{remark}
\label{Rem1} The entries of the response vector
$(r_0^N,r_1^N,\ldots,r_{2N-2}^N)$) depends on
$\{a_0,\ldots,a_{N-1}\}$, $\{b_1,\ldots,b_{N}\}$, and does not
depend on the boundary condition at $n=N+1$, the entries starting
from $r_{2N-1}^N$ does "feel" the boundary condition at $n=N+1$.
\end{remark}

On introducing the special control $\delta=(1,0,0,\ldots)$, one
can see that the kernel of the response operator (\ref{R_def_int})
is given by
\begin{equation}
\label{con1}
r_{t-1}^{N}=\left(R^T_{N}\delta\right)_t=v^\delta_{1,\,t},\quad
t=1,\ldots.
\end{equation}
The spectral function of operator $A^N$ is introduced by the rule
\begin{equation}
\label{measure}
\rho^{N}(\lambda)=\sum_{\{k\,|\,\lambda_k<\lambda\}}\frac{1}{\rho_k},
\end{equation}
then from (\ref{Jac_sol_rep_int}), (\ref{con1}) we immediately
deduce
\begin{proposition}
The solution to (\ref{Jacobi_dyn_int}), the response vector of
(\ref{Jacobi_dyn_int}) and entries of the matrix of the connecting
operator $C^T_N$ admit the following spectral representations:
\begin{eqnarray}
v^f_{n,t}=\int_{-\infty}^\infty \sum_{k=1}^t
\mathcal{T}_k(\lambda)f_{t-k}\varphi_n(\lambda)\,d\rho^{N}(\lambda),\quad n,t\in \mathbb{N},\label{Jac_sol_spectr} \\
r_{t-1}^{N}=\int_{-\infty}^\infty
\mathcal{T}_t(\lambda)\,d\rho^{N}(\lambda),\quad t\in \mathbb{N},\label{Resp_int_spectr}\\
\{C^T_{N}\}_{l+1,\,m+1}=\int_{-\infty}^\infty
\mathcal{T}_{T-l}(\lambda)\mathcal{T}_{T-m}(\lambda)\,d\rho^{N}(\lambda),
\,\, l,m=0,\ldots,T-1\label{SP_mes_int} .
\end{eqnarray}
\end{proposition}

\subsection{Semi-infinite Jacobi matrix}

We consider an initial boundary value problem for a dynamical
system with discrete time associated with a semi-infinite Jacobi
matrix $A$:
\begin{equation}
\label{Jacobi_dyn} \left\{
\begin{array}l
u_{n,\,t+1}+u_{n,\,t-1}-a_{n}u_{n+1,\,t}-a_{n-1}u_{n-1,\,t}-b_nu_{n,\,t}=0,\quad n\in \mathbb{N},\, t\in \mathbb{N}_0,\\
u_{n,\,-1}=u_{n,\,0}=0,\quad n\in \mathbb{N}, \\
u_{0,\,t}=f_t,\quad t\in \mathbb{N}_0,
\end{array}\right.
\end{equation}
which is a discrete analog of an initial boundary value problem
for a wave equation with a potential on a half-line with the
Dirichlet control at zero. The solution to (\ref{Jacobi_dyn}) is
denoted by $u^f_{n,\,t}$. We fix some positive integer $T$ and
denote by $\mathcal{F}^T$ the \emph{outer space} of the system
(\ref{Jacobi_dyn}), the space of controls (inputs):
$\mathcal{F}^T:=\mathbb{R}^T$, $f\in \mathcal{F}^T$,
$f=(f_0,\ldots,f_{T-1})$.
\begin{lemma}
A solution to (\ref{Jacobi_dyn}) admits the representation
\begin{equation}
\label{Jac_sol_rep} u^f_{n,\,t}=\prod_{k=0}^{n-1}
a_kf_{t-n}+\sum_{s=n}^{t-1}w_{n,\,s}f_{t-s-1},\quad n,t\in
\mathbb{N},
\end{equation}
where $w_{n,s}$ satisfies certain Goursat problem.
\end{lemma}
The input $\longmapsto$ output correspondence in the system
(\ref{Jacobi_dyn}) is realized by a \emph{response operator}:
$R^T:\mathcal{F}^T\mapsto \mathbb{R}^T$ defined by the rule
\begin{equation*}
\left(R^Tf\right)_t=u^f_{1,\,t}, \quad t=1,\ldots,T.
\end{equation*}
This operator plays the role of inverse data, the corresponding
inverse problem was considered in \cite{MM,MM3}. The convolution
kernel of $R^T$ is called a \emph{response vector}, in accordance
with (\ref{Jac_sol_rep}) one has that
$r=(r_0,r_1,\ldots,r_{T-1})=(1,w_{1,1},w_{1,2},\ldots w_{1,T-1})$:
\begin{eqnarray}
\label{R_def}
\left(R^Tf\right)_t=u^f_{1,\,t}=f_{t-1}+\sum_{s=1}^{t-1}
w_{1,\,s}f_{t-1-s}
\quad t=1,\ldots,T.\\
\notag \left(R^Tf\right)=r*f_{\cdot-1}.
\end{eqnarray}
 By choosing a special control
 $f=\delta=(1,0,0,\ldots)$, the kernel of the response operator can be determined as
\begin{equation*}
\left(R^T\delta\right)_t=u^\delta_{1,\,t}= r_{t-1}.
\end{equation*}

For a fixed $T\in \mathbb{N}$ we introduce the \emph{inner space}
of the dynamical system (\ref{Jacobi_dyn})
$\mathcal{H}^T:=\mathbb{R}^T$, $h\in \mathcal{H}^T$,
$h=(h_1,\ldots, h_T)$, the space of states. The wave
$u^f_{\cdot,\,T}$ is considered as a state of the system
(\ref{Jacobi_dyn}) at the moment $t=T$. By (\ref{Jac_sol_rep}) we
have that $u^f_{\cdot,\,T}\in \mathcal{H}^T$. The input
$\longmapsto$ state correspondence of the system
(\ref{Jacobi_dyn}) is realized by a \emph{control operator}
$W^T:\mathcal{F}^T\mapsto \mathcal{H}^T$, defined by the rule
\begin{equation*}
\left(W^Tf\right)_n:=u^f_{n,\,T},\quad n=1,\ldots,T.
\end{equation*}
From (\ref{Jac_sol_rep}) we deduce the representation for $W^T$:
\begin{equation*}
\left(W^Tf\right)_n=u^f_{n,\,T}=\prod_{k=0}^{n-1}
a_kf_{T-n}+\sum_{s=n}^{T-1}w_{n,\,s}f_{T-s-1},\quad n=1,\ldots,T.
\end{equation*}
Or in matrix form:
\begin{equation}
\label{WT}
W^Tf=\begin{pmatrix} u_{1,\,T}\\
u_{2,\,T}\\
\cdot\\
u_{k,\,T}\\
\cdot\\
u_{T,\,T}
\end{pmatrix}=\begin{pmatrix}
w_{1,T-1}& w_{1,T-2} & w_{1,T-3} & \ldots & \ldots & 1\\
w_{2,T-1} & w_{2,T-1} & \ldots & \ldots & a_1 & 0\\
\cdot & \cdot & \cdot & \cdot & \cdot & \cdot \\
 w_{k, T-1} & \ldots & \prod_{j=0}^{k-1}
a_j& 0 & \ldots &0\\
\cdot & \cdot & \cdot & \cdot & \cdot & \cdot \\
\prod_{k=0}^{T-1} a_{k} & 0 & 0 & 0 & \ldots & 0
\end{pmatrix}
\begin{pmatrix} f_{0}\\
f_{1}\\
\cdot\\
f_{T-k-1}\\
\cdot\\
f_{T-1}
\end{pmatrix}.
\end{equation}

The following statement is equivalent to a boundary
controllability of (\ref{Jacobi_dyn}):
\begin{lemma}
\label{teor_control} The operator $W^T$ is an isomorphism between
$\mathcal{F}^T$ and $\mathcal{H}^T$.
\end{lemma}

We introduce the \emph{connecting operator} $C^T:
\mathcal{F}^T\mapsto \mathcal{F}^T$ for the system
(\ref{Jacobi_dyn}), by the quadratic form: for arbitrary $f,g\in
\mathcal{F}^T$ we define
\begin{equation}
\label{C_T_def} \left(C^T
f,g\right)_{\mathcal{F}^T}=\left(u^f_{\cdot,\,T},
u^g_{\cdot,\,T}\right)_{\mathcal{H}^T}=\left(W^Tf,W^Tg\right)_{\mathcal{H}^T}.
\end{equation}
That is $C^T=\left(W^T\right)^*W^T$. The fact that the connecting
operator can be represented in terms of inverse data is crucial in
the BC method.
\begin{theorem}
The connecting operator $C^T$ is an isomorphism in
$\mathcal{F}^T$, it admits the representation in terms of dynamic
inverse data:
\begin{equation}
\label{C_T_repr} C^T=C^T_{ij},\quad
C^T_{ij}=\sum_{k=0}^{T-\max{i,j}}r_{|i-j|+2k},\quad r_0=a_0=1.
\end{equation}
\begin{equation*}
C^T=
\begin{pmatrix}
r_0+r_2+\ldots+r_{2T-2} & r_1+\ldots+r_{2T-3} & \ldots &
r_T+r_{T-2} &
r_{T-1}\\
r_1+r_3+\ldots+r_{2T-3} & r_0+\ldots+r_{2T-4} & \ldots & \ldots
&r_{T-2}\\
\cdot & \cdot & \cdot & \cdot & \cdot \\
r_{T-3}+r_{T-1}+r_{T+1} &\ldots & r_0+r_2+r_4 & r_1+r_3 & r_2\\
r_{T}+r_{T-2}&\ldots &r_1+r_3&r_0+r_2&r_1 \\
r_{T-1}& r_{T-2}& \ldots & r_1 &r_0
\end{pmatrix}.
\end{equation*}
\end{theorem}

One can observe \cite{MM} that $C^T_{ij}$ satisfies the difference
boundary problem:
\begin{corollary}
\label{Corl} The kernel of $C^T$ satisfies
\begin{equation*}
\left\{
\begin{array}l
C^T_{i,j+1}+C^T_{i,j-1}-C^T_{i+1,j}-C^T_{i-1,j}=0,\\
C^T_{i,T}=r_{T-i},\,\,C^T_{T,j}=r_{T-j},\,\, r_0=1.
\end{array}
\right.
\end{equation*}
\end{corollary}

With the matrix $A$ we associate the operator $A$ (we keep the
same notation), defined on $l^2\ni \phi=(\phi_1,\phi_2,\ldots)$,
and given by
\begin{align*}
(A\phi)_n&=a_{n}\phi_{n+1}+a_{n-1}\phi_{n-1}+b_n\phi_n,\quad
n\geqslant 2,\\ 
(A\phi)_1&=b_1\phi_1+a_1\phi_2,\quad n=1. 
\end{align*}
By $d\rho(\lambda)$ we denote the spectral measure of $A$
(non-unique if $A$ is in the limit circle case at infinity), see
\cite{Ahiez,A}.

The Remark \ref{Rem1} in particular implies that
\begin{eqnarray}
\label{R_eqv} R^{2N-2} = R^{2N-2}_{N},\\
\label{Rav_bc} u^f_{n,\,t}=v^f_{n,\,t},\quad n\leqslant t\leqslant
N,\quad \text{and}\quad W^N=W^N_{N}.
\end{eqnarray}
Thus due to (\ref{R_eqv}), we have that $r_{t-1}=r_{t-1}^{N}$,
$t=1,\ldots,2N$. On the other hand, taking into the account
(\ref{Rav_bc}), we can see that $C^T=C^T_{N}$ with $T\leqslant N$.
Thus, taking in (\ref{Resp_int_spectr}), (\ref{SP_mes_int})
$N\to\infty$, we obtain the
\begin{proposition}
The entries of the response vector of (\ref{Jacobi_dyn}) and of
the matrix of the connecting operator $C^T$ admit the spectral
representation:
\begin{eqnarray}
r_{t-1}=\int_{-\infty}^\infty
\mathcal{T}_t(\lambda)\,d\rho(\lambda),\quad t\in
\mathbb{N},\label{Resp_spectr}\\
\{C^T\}_{l+1,\,m+1}=\int_{-\infty}^\infty
\mathcal{T}_{T-l}(\lambda)\mathcal{T}_{T-m}(\lambda)\,d\rho(\lambda),
\quad l,m=0,\ldots,T-1. \label{SP_mes_d}
\end{eqnarray}
\end{proposition}

\section{Truncated moment problem. Recovering "Dirichlet" spectral data from dynamic data.}

We make the following observation: in the moment problem one
recovers the measure from the given set of moments
(\ref{Moment_eq}), in the dynamic inverse problem \cite{MM,MM3}
one recovers the Jacobi matrix from the given response vector
(\ref{R_def}). The spectral representation of response vector
(\ref{Resp_spectr}),the results from the previous section and from
\cite{MM,MM3} implies the following
\begin{remark}
The knowledge of the finite set of moments
$\{s_0,s_1,\ldots,s_{2N-2}\}$ is equivalent to the knowledge of
$\{r_0,r_1,\ldots,r_{2N-2}\}$, wrom which it is possible to
recover Jacobi matrix $A^N\in \mathbb{R}^{N\times N}$ whose
elements can be thought of as a coefficients in dynamical system
(\ref{Jacobi_dyn_int}) with Dirichlet boundary condition at
$n=N+1$, or $N\times N$ block in semi-infinite Jacobi matrix in
(\ref{Jacobi_dyn_int}) with no condition at the right end.
\end{remark}

\begin{definition}
By a solution of a truncated moment problem of order $N$ we call a
Borel measure $d\rho(\lambda)$ on $\mathbb{R}$ such that
equalities (\ref{Moment_eq}) with this measure hold for
$k=0,1,\ldots,2N-2.$
\end{definition}

In \cite{MM3} the authors proved the following
\begin{theorem}
\label{Th_char} The vector $(r_0,r_1,r_2,\ldots,r_{2N-2})$ is a
response vector for the dynamical system (\ref{Jacobi_dyn_int}) if
and only if the matrix $C^N$ defined by (\ref{SP_mes_d}),
(\ref{C_T_repr}) is positive definite.
\end{theorem}
This theorem and formulas for the entries of Jacobi matrix, which
will be provided later, imply the following procedure of solving
the truncated moment problem:
\begin{itemize}
\item[1)] Calculate $(r_0,r_1,r_2,\ldots,r_{2N-2})$ from
$\{s_0,s_1,\ldots,s_{2N-2}\}$ by using (\ref{Resp_spectr}).

\item[2)] Recover $N\times N$ Jacobi matrix $A^N$ using formulas
for $a_k,$ $b_k$ from \cite{MM3}

\item[3)] Recover spectral measure for finite Jacobi matrix $A^N$
with prescribed arbitrary selfadjoint condition at $n=N+1$.

\item[3')] Extend Jacobi matrix $A^N$ to \emph{finite} Jacobi
matrix $A^M$, $M>N$, prescribe arbitrary selfadjoint condition at
$n=M+1$ and recover spectral measure of $A^M$.

\item[3'')] Extend Jacobi matrix $A^N$ to \emph{infinite} Jacobi
matrix $A$, and recover spectral measure of $A$.
\end{itemize}
Every measure obtained in $3),$ $3')$, $3'')$ gives a solution to
the truncated moment problem. Below we propose a different
approach: using the ideas of the BC method we recover the spectral
measure corresponding to Jacobi matrix $A^N$ directly from moments
(from the operator $C^N$), without recovering the Jacoi matrix
itself.

\begin{agreement}
\label{agr} We assume that controls $f\in \mathcal{F}^N$,
$f=\left(f_0,\ldots,f_{N-1}\right)$ are extended:
$f=\left(f_{-1},f_0,\ldots,f_{N-1},f_N\right)$, where
$f_{-1}=f_N=0$.
\end{agreement}
We introduce the special space of controls
$\mathcal{F}^N_0=\left\{f\in \mathcal{F}^T\,|\, f_0=0\right\}$ and
the operator $D: \mathcal{F}^N\mapsto \mathcal{F}^N$ acting by
\begin{equation*}
\left(Df\right)_t=f_{t+1}+f_{t-1}.
\end{equation*}
The following statement can be easily proved using arguments from
\cite{MM3} and representations (\ref{WT}) and (\ref{Jac_sol_rep}):
\begin{proposition}
\label{PropContr}
\begin{itemize}

\item[1)]The operator $W^N$ maps $\mathcal{F}^N_0$ isomorphically
onto $\mathcal{H}^{N-1}$.

\item[2)]On the set $\mathcal{F}^N_0$ the following relation
holds:
\begin{equation}
\label{FT_proper} W^NDf=DW^Nf,\quad f\in \mathcal{F}^N_0.
\end{equation}
\end{itemize}
\end{proposition}
Taking $f,g\in \mathcal{F}_0^N$ and evaluating the quadratic form,
bearing in mind (\ref{FT_proper}), we obtain:
\begin{eqnarray}
\label{L_quadr}
\left(C^NDf,f\right)_{\mathcal{F}^N}=\left(W^NDf,W^Ng\right)_{\mathcal{H}^{N}}=
\left(DW^Nf,W^Ng\right)_{\mathcal{H}^N}\\
=\left(A^{N-1}v^f,v^g\right)_{\mathcal{H}^N}.\notag
\end{eqnarray}
The last equality in (\ref{L_quadr}) means that only $A^{N-1}$
block of the whole matrix $A^N$ is in use.  Then it is possible to
perform the spectral analysis of $A^{N-1}$ using the classical
variational approach, the controllability of the system
(\ref{Jacobi_dyn_int}) (see Proposition \ref{PropContr}) and the
representation (\ref{L_quadr}), see also \cite{B2001}. The
spectral data of Jacobi matrix $A^{N-1}$ with the Dirichlet
boundary condition at $n=N$ can be recovered by the following
procedure:

\begin{itemize}
\item[1)] The first eigenvalue is given by
\begin{equation}
\label{M1}\lambda_1^{N-1}=\min_{f\in \mathcal{F}^N_0,\,(C^N
f,f)_{\mathcal{F}^N}=1}\left(C^N Df,f\right)_{\mathcal{F}^N}.
\end{equation}

\item[2)] Let $f^1$, be the minimizer of (\ref{M1}), then
\begin{equation*}
\rho_1=\left(C^N f^1,f^1\right)_{\mathcal{F}^N}.
\end{equation*}

\item[3)] The second eigenvalue is given by
\begin{equation}
\label{M2}\lambda_2^{N-1}=\min_{\substack {f\in
\mathcal{F}^N_0,(C^N f,f)_{\mathcal{F}^N}=1\\
(C^N f,f_l)_{\mathcal{F}^N}=0}}\left(C^N
Df,f\right)_{\mathcal{F}^N}.
\end{equation}

\item[4)] Let $f^2$, be the minimizer of (\ref{M2}), then
\begin{equation*}
\rho_2=\left(C^N f^2,f^2\right)_{\mathcal{F}^N}.
\end{equation*}
\end{itemize}
Continuing this procedure, one recovers the set
$\{\lambda_k^{N-1},\rho_k\}_{k=1}^{N-1}$ and constructs the
measure $d\rho^{N-1}(\lambda)$ by (\ref{measure}).

\begin{remark}
The measure, constructed by the above procedure solves the
truncated moment problem for the set of moments
$\{s_0,s_1,\ldots,s_{2N-4}\}$.
\end{remark}

\subsection{Euler-Lagrange equations}

In this section we derive equations which can be thought of as a
Euler-Lagrange equations for the problem of the minimization of a
functional $\left(C^NDf,f\right)_{\mathcal{F}^N}$ in
$\mathcal{F}^N_0$ with the constrain
$\left(C^Tf,f\right)_{\mathcal{F}^N}=1$, described in the previous
section. Similar method of deriving equations which can be used
for recovering of spectral data was used in \cite{AMM}.

By $f^k$, $k=1,\ldots,N$ we denote the control that drive system
(\ref{Jacobi_dyn_int}) to prescribed state (see (\ref{Phi_def})):
\begin{equation*}
W^Tf_k=\phi^k,\quad k=1,\ldots,N.
\end{equation*}
Due to Proposition \ref{PropContr}, such a control exists and is
unique for every $k$. We introduce the operator
\begin{gather*}
D_-^N: \mathcal{F}^N\mapsto \mathcal{F}^{N},\\
\left(D^N_-f\right)_n=f_{n-1},\quad n=1,\ldots,N-1,\quad
\left(D^N_-f\right)_0=0,
\end{gather*}
and denote by $P: \mathcal{F}^{N+1}\mapsto \mathcal{F}^{N}$ the
embedding operator. Then $P^*:\mathcal{F}^{N}\mapsto
\mathcal{F}^{N+1}$ extends vector by zero:
$\left(P^*f\right)_{k}=f_k$, $k=0,1,N-1$, and
$\left(P^*f\right)_{N}=0.$

\begin{theorem}
\label{teor} The spectrum of $A^{N}$ and (non-normalized) controls
$f_k$, $k=1,\ldots,N$ are the spectrum and the eigenvectors of the
following generalized spectral problem:
\begin{equation}
\label{m_eqn}
\left(P\left(D^{N+1}_-\right)^*C^{N+1}P^*+C^ND^N_-\right)f_k=\lambda_kC^Nf_k,\quad
k=1,\ldots,N.
\end{equation}
\end{theorem}
\begin{proof}
For $h\in \mathcal{F}^T$ we always assume that $h_{-1}=h_T=0$ (see
Agreement \ref{agr}). For a fixed $k=1,\ldots,N$ we take $f_k\in
\mathcal{F}^N$ such that $W^Nf_k=v^{f_k}_{\cdot,\,N}=\phi^k$, then
for arbitrary $g\in \mathcal{F}^N$ we can evaluate:
\begin{eqnarray}
\left(\lambda_kC^N
f_k,g\right)_{\mathcal{F}^N}=\left(\lambda_kv^{f_k}_{\cdot,\,N},v^g_{\cdot,\,N}\right)_{\mathcal{H}^N}=
\left(\lambda_k\phi^{k},v^g_{\cdot,\,N}\right)_{\mathcal{H}^N}=
\left(A^{N}\phi^{k},v^g_{\cdot,\,N}\right)_{\mathcal{H}^N}\notag\\
=\left(\left(A^{N}v^{f_k}\right)_{\cdot,\,N},v^g_{\cdot,\,N}\right)_{\mathcal{H}^N}
=\left(\left(Dv^{f_k}\right)_{\cdot,\,N},v^g_{\cdot,\,N}\right)_{\mathcal{H}^N}\notag\\
=\left(v^{f_k}_{\cdot,\,N+1},v^g_{\cdot,\,N})\right)_{\mathcal{H}^N}+\left(v^{f_k}_{\cdot,\,N-1},v^g_{\cdot,\,N}\right)_{\mathcal{H}^N}.\label{C_ev1}
\end{eqnarray}
We note that
\begin{eqnarray*}
\mathcal{H}^N \ni
v^g_{\cdot,\,N}=\left(v^g_{1,\,N},\ldots,v^g_{N,\,N}\right),\\
\mathcal{H}^{N+1} \ni v^{D^{N+1}_-
g}_{\cdot,\,N+1}=\left(v^g_{1,\,N},\ldots,v^g_{N,\,N},0\right).
\end{eqnarray*}
That is why we can rewrite the first summand in the right hand
side of (\ref{C_ev1}) as
\begin{equation}
\label{C_ev2}
\left(v^{f_k}_{\cdot,\,N+1},v^g_{\cdot,N}\right)_{\mathcal{H}^N}=\left(v^{f_k}_{\cdot,\,N+1},v^{D^{N+1}_-
g}_{\cdot,\,N+1}\right)_{\mathcal{H}^{N+1}}=\left(C^{N+1}f_k,D^{N+1}_-g\right)_{\mathcal{F}^{N+1}}.
\end{equation}
Analogously:
\begin{eqnarray*}
\mathcal{H}^{N} \ni
v^{f_k}_{\cdot,\,N-1}=\left(v^{f_k}_{1,\,N-1},\ldots,v^{f_k}_{N-1,N-1},0\right),\\
\mathcal{H}^{N} \ni
v^{D^N_-f_k}_{\cdot,\,N}=\left(v^{f_k}_{1,\,N-1},\ldots,v^{f_k}_{N-1,N-1},0\right).
\end{eqnarray*}
So we can rewrite the second summand in the right hand side of
(\ref{C_ev1}) as
\begin{equation}
\label{C_ev3}
\left(v^{f_k}_{\cdot,\,N-1},v^g_{\cdot,\,N}\right)_{\mathcal{H}^N}=
\left(v^{D^N_-f_k}_{\cdot,\,N},v^{g}_{\cdot,\,N}\right)_{\mathcal{H}^{N}}=\left(C^{N}D^N_-f_k,g\right)_{\mathcal{F}^{N}}.
\end{equation}
Finally from (\ref{C_ev1}), (\ref{C_ev2}) and (\ref{C_ev3}) we
deduce that
\begin{equation}
\label{Eq1} \left(\lambda_kC^N
f_k,g\right)_{\mathcal{F}^N}=\left(C^{N+1}f_k,D^{N+1}_-g\right)_{\mathcal{F}^{N+1}}+\left(C^{N}D^N_-f_k,g\right)_{\mathcal{F}^{N}}.
\end{equation}
Using operators $P,$ $P^*$  we can rewrite (\ref{Eq1}) in the form
\begin{equation*}
\left(P\left(D^{N+1}_-\right)^*C^{N+1}P^*+C^ND^N_-\right)f_k=\lambda_kC^Nf_k.
\end{equation*}

Thus the pair $\{f_k,\lambda_k\}$ gives the solution to
(\ref{m_eqn}). Now let the pair $\{f,\lambda\}$ be the solution to
(\ref{m_eqn}) with $f\in \mathcal{F}^N$ $f\not= f_k,$
$\lambda\not=\lambda_k$ for any $k=1\,\ldots,N$. Then
$W^Nf=v^f_{\cdot,\,N}=\sum_{k=1}^{N} a_k \phi^k$ for some $a_k\in
\mathbb{R}$. We can evaluate for arbitrary $g\in \mathcal{F}^N$:
\begin{eqnarray*}
0=\left(\left(P\left(D^{N+1}_-\right)^*C^{N+1}P^*+C^ND^N_-\right)f-\lambda
C^Nf,g\right)_{\mathcal{F}^N}\\
=\left(C^{N+1}P^*f,D^{N+1}_-P^*g\right)_{\mathcal{F}^{N+1}}
+\left(C^N D^N_-f,g\right)_{\mathcal{F}^N}-\lambda\left(
v^f_{\cdot,\,N},v^g_{\cdot,\,N}\right)_{\mathcal{H}^N} \\
=\left(
v^{P^*f}_{\cdot,\,N+1},v^{D^{N+1}_-P^*g}_{\cdot,\,N+1}\right)_{\mathcal{H}^{N+1}}+
\left(v^{D^N_-f}_{\cdot,\,N},v^g_{\cdot,\,N}\right)_{\mathcal{H}^N}
-\lambda\left(v^f_{\cdot,\,N},v^g_{\cdot,\,N}\right)_{\mathcal{H}^N}\\
= \left(
v^{P^*f}_{\cdot,\,N+1},v^{g}_{\cdot,\,N+1}\right)_{\mathcal{H}^{N}}+\left(v^{f}_{\cdot,\,N-1},v^g_{\cdot,\,N-1}\right)_{\mathcal{H}^N}
-\lambda\left(v^f_{\cdot,\,N-1},v^g_{\cdot,\,N-1}\right)_{\mathcal{H}^N}\\
=\left(\left(A^Nv^{f}\right)_{\cdot,\,N},v^{g}_{\cdot,N}\right)_{\mathcal{H}^{N}}-\lambda\left(
v^f_{\cdot,\,N},v^g_{\cdot,\,N}\right)_{\mathcal{H}^N}\\
=\left( A^N\sum_{k=1}^{N} a_k \phi^k-\lambda \sum_{k=1}^{N} a_k
\phi^k,W^Ng\right)_{\mathcal{H}^N}= \left( \sum_{k=1}^{N} a_k
(\lambda_k-\lambda)\phi^k,W^Ng\right)_{\mathcal{H}^N}.
\end{eqnarray*}
From the above equality and Proposition \ref{PropContr} it follows
that all $a_k$ except one are equal to zero, and for such $a_j,$
$\lambda=\lambda_j$, which completes the proof.
\end{proof}

Having found spectrum and non-nomalized controls from
(\ref{m_eqn}) we can recover the measure of operator $A^N$ with
Dirichlet boundary condition at $n=N+1$ by the following
procedure:
\begin{itemize}

\item[1)] Normalize controls by choosing
$\left(C^Nf_k,f_k\right)_{\mathcal{F}^N}=1$,

\item[2)] Observe that $W^Nf^k=\alpha_k\phi^k$ for some
$\alpha_k\in \mathbb{R}$, where the constant is defined by
$\alpha_k=\left(Rf_k\right)_N$.

\item[3)] The norming coefficients are given by
$\rho_k={\alpha_k^2}$, $k=1,\ldots,N$.

\item[4)] Recover the measure by (\ref{measure}).

\end{itemize}

Now we rewrite the generalized spectral problem (\ref{m_eqn}) in
more details and transfer the matrices in left and right-hand
sides to Hankel matrices known from classical literature
\cite{Ahiez,S}. 
Note that the matrices in (\ref{m_eqn}) has the following
representations:
\begin{align}
P=\begin{pmatrix} 1 & 0 & \ldots & 0\\
0 & 1 & \ldots & 0\\
\cdot & \cdot & \ldots & \cdot \\
0& \ldots & 1 & 0
\end{pmatrix},\quad
P^*=\begin{pmatrix} 1 & 0 & \ldots & 0\\
0 & 1 & \ldots & 0\\
\cdot & \cdot & \ldots & \cdot \\
0& \ldots & 0& 1\\
0& \ldots & 0& 0
\end{pmatrix},\notag\\
D^N_-=\begin{pmatrix} 0 & 0 & \ldots & 0\\
1 & 0 & \ldots & 0\\
\cdot & \cdot & \ldots & \cdot \\
0& \ldots & 1 & 0
\end{pmatrix},\quad
\left(D^N_-\right)^*= \begin{pmatrix} 0 & 1 & \ldots & 0\\
0 & 0 & \ldots & 0\\
\cdot & \cdot & \ldots & \cdot \\
0& \ldots & 0 & 1\\
 0& \ldots & 0 & 0
\end{pmatrix},\notag \\
C^N=\begin{pmatrix} c_{N,N} & \ldots & \cdot & c_{N,1}\\
c_{N-1,N} &  & \cdot & c_{N-1,1}\\
\cdot & \cdot & \cdot & \cdot \\
c_{1,N}& \ldots & \cdot & c_{1,1}
\end{pmatrix},\, C^{N+1}=\begin{pmatrix} c_{N+1,N+1} & \ldots & \cdot & c_{N+1,1}\\
c_{N,N+1} &  & \cdot & c_{N,1}\\
\cdot & \cdot & \cdot & \cdot \\
c_{1,N+1}& \ldots & \cdot & c_{1,1}
\end{pmatrix}. \label{CT_newdef}
\end{align}
Here we used the notations for entries of $C^N$ different from
ones in (\ref{C_T_repr}) in order to show that $C^N$ is a lower
right block in $C^{N+1}$. The left hand side of (\ref{m_eqn}) we
denote by
\begin{equation*}
B^N:=P\left(D^{N+1}_-\right)^*C^{N+1}P^*+C^ND^N_-.
\end{equation*}

\begin{proposition}
The matrix $B^N$ is self-adjoint, it admits the following
representation:
\begin{equation}
\label{BN_def}
B^N=\begin{pmatrix} c_{N,N+1}+c_{N,N-1} & c_{N,N}+c_{N,N-2} & \ldots & c_{N,3}+c_{N,1} & c_{N,2}\\
c_{N-1,N+1}+c_{N-1,N-1} & \ldots & \ldots & c_{N-1,3}+c_{N-1,1} & c_{N-1,2}\\
\cdot & \cdot & \ldots & \cdot & \cdot \\
c_{1,N+1}+c_{1,N-1}& c_{1,N}+c_{1,N-2} & \ldots & \ldots & c_{1,2}
\end{pmatrix}.
\end{equation}
\end{proposition}
\begin{proof}
We note that the matrices $P$, $P^{*}$  and $D_-$, $D^{*}_-$
have one diagonal filled with ones  and the other elements are
zeros.  Thus the multiplication by such a matrix leads to deleting
a line or column from the original matrix (possibly with the
addition of a zero line or column). Performing calculations we see
that the first term in the right hand side of $B^N$  is obtained
by deleting last column and first row from $C^{N+1}$ and the
second term is obtained by deleting the first column and adding
zero column to $C^N$. All aforesaid leads to the formula
(\ref{BN_def}).

We note that the above representation (\ref{BN_def}) and Corollary
\ref{Corl}  shows that $B^T$ is self-adjoint matrix.
\end{proof}

\begin{remark}
The spectral problem (\ref{m_eqn}) has a form
\begin{equation}
\label{m_eqn1} B^Nf_k=\lambda_k C^Nf_k,\quad k=1,\ldots,N,\quad
\quad\text{where}\quad C^N>0,\quad B^N=\left(B^N\right)^*.
\end{equation}
\end{remark}

Chebyshev polynomials of the second kind $\{
\mathcal{T}_1(\lambda),\mathcal{T}_2(\lambda),\ldots
\mathcal{T}_n(\lambda)\}$ (see \ref{Chebysh}) are related to $\{
1,\lambda,\lambda^{n-1}\}$ by the following relation
\begin{equation}
\label{Response_moments_rel}
\begin{pmatrix}
\mathcal{T}_1(\lambda)\\
\mathcal{T}_2(\lambda)\\
\ldots \\
\mathcal{T}_n(\lambda)
\end{pmatrix}=\Lambda_n \begin{pmatrix}
1\\
\lambda\\
\ldots \\
\lambda^{n-1}
\end{pmatrix}=\begin{pmatrix}
1& 0& \ldots & 0\\
a_{21}& 1 & \ldots & 0\\
\ldots \\
a_{n1} & a_{n2}& \ldots & 1
\end{pmatrix}\begin{pmatrix}
1\\
\lambda\\
\ldots \\
\lambda^{n-1}
\end{pmatrix}.
\end{equation}
\begin{proposition}
The entries of the matrix $\Lambda_n\in R^{n-1\times n-1}$ are
given by
\begin{equation}
\label{LambdaMatr}
\Lambda_n=a_{ij}=\begin{cases} 0,\quad \text{if $i>j$},\\
0,\quad \text{if $i+j$ is odd,}\\
C_{\frac{i+j}{2}}^j(-1)^{\frac{i+j}{2}+j}.
\end{cases}
\end{equation}
The entries of the response vector are related to moments by the
rule:
\begin{equation}
\label{Resp_and_moments}
\begin{pmatrix}
r_0\\
r_1\\
\ldots \\
r_{n-1}
\end{pmatrix}=\Lambda_n\begin{pmatrix}
s_0\\
s_1\\
\ldots \\
s_{n-1}
\end{pmatrix}.
\end{equation}
\end{proposition}
\begin{proof}
The formula (\ref{LambdaMatr}) for entries of $\Lambda_n$ is
proved by direct calculations with the use of properties of
Chebyshev polynomials. Then making use of (\ref{Resp_spectr})
yields (\ref{Resp_and_moments}).
\end{proof}

Introduce the following Hankel matrices
\begin{equation*}
S^N_{m}:=\begin{pmatrix} s_{2N-2+m} & s_{2N-3+m} & \ldots & s_{N-1+m}\\
s_{2N-3+m} & \ldots & \ldots & \ldots\\
\cdot & \cdot & \ldots & s_{1+m} \\
s_{N-1+m} & \ldots & s_{1+m} & s_{m}
\end{pmatrix},\quad m=0,1,\dots,
\end{equation*}
the matrix $J_N\in \mathbb{R}^{N\times N}$:
\begin{equation*}
J_N=\begin{pmatrix} 0 & \ldots & 0 & 1\\
0 & \ldots & 1 & 0\\
\cdot & \cdot & \ldots & \cdot \\
0& 1 & \ldots & 0\\
 1& \ldots & 0 & 0
\end{pmatrix},\quad J_NJ_N=I_N=\begin{pmatrix} 1 & 0 & \ldots & 0\\
0 & 1 & \ldots & 0\\
\cdot & \cdot & \ldots & \cdot \\
0& \ldots & 1 & 0\\
0& \ldots & 0 & 1
\end{pmatrix},
\end{equation*}
and define
\begin{equation*}
\widetilde\Lambda_N:=J_N\Lambda_N J_N.
\end{equation*}
The remarkable fact is that the matrices $B^N,\,C^N$ can be
reduced to Hankel matrices by the same linear transformation:
\begin{theorem}
\label{Propos_EL} The following relations hold:
\begin{eqnarray}
\label{CT_S0} C^N=\widetilde\Lambda_N S_0^N
\left(\widetilde\Lambda_N\right)^*,\\
\label{BT_S1} B^N =\widetilde\Lambda_N S_1^N
\left(\widetilde\Lambda_N\right)^*.
\end{eqnarray} Then the
generalized spectral problem (\ref{m_eqn}) or (\ref{m_eqn1}) upon
introducing the notation
$g_k=\left(\widetilde\Lambda_N\right)^*f_k$ is equivalent to the
following generalized spectral problem:
\begin{equation}
\label{m_eqn2} S^N_1g_k=\lambda_k S^N_0g_k.
\end{equation}
\end{theorem}
\begin{proof}
Using (\ref{CT_newdef}) and the representation (\ref{SP_mes_d}),
we have that entries of $C^N$ have a form:
\begin{equation*}
c_{ij}=\int_{-\infty}^\infty
\mathcal{T}_{N-i+1}(\lambda)\mathcal{T}_{N-j+1}(\lambda)\,d\rho(\lambda),\quad
i,j=i,\ldots,N.
\end{equation*}
We can write down the form of the operator $C^N$ then:
\begin{equation}
\label{CN_new} C^N=\int_{-\infty}^\infty
\begin{pmatrix}\mathcal{T}_{N}(\lambda)\\\mathcal{T}_{N-1}(\lambda)\\\cdot\\\mathcal{T}_1(\lambda)\end{pmatrix}\otimes
\begin{pmatrix}\mathcal{T}_{N}(\lambda)\\\mathcal{T}_{N-1}(\lambda)\\\cdot\\\mathcal{T}_1(\lambda)\end{pmatrix}\,d\rho(\lambda).
\end{equation}
Using (\ref{Response_moments_rel}) we can rewrite (\ref{CN_new})
as
\begin{eqnarray*}
C^N=\int_{-\infty}^\infty J_N\Lambda_N J_NJ_N\begin{pmatrix} 1\\
\lambda\\ \cdot\\ \lambda^{N-1}\end{pmatrix}\otimes
J_N\Lambda_N J_NJ_N\begin{pmatrix} 1\\
\lambda\\ \cdot\\ \lambda^{N-1}\end{pmatrix}\,d\rho(\lambda)\\
=\int_{-\infty}^\infty \widetilde\Lambda_N\begin{pmatrix} \lambda^{N-1}\\
\lambda^{N-2}\\ \cdot\\ 1\end{pmatrix}\otimes
\widetilde\Lambda_N \begin{pmatrix} \lambda^{N-1}\\
\lambda^{N-2}\\ \cdot\\ 1\end{pmatrix}\,d\rho(\lambda)\\
=\widetilde\Lambda_N \int_{-\infty}^\infty \begin{pmatrix} \lambda^{N-1}\\
\lambda^{N-2}\\ \cdot\\ 1\end{pmatrix}\otimes \begin{pmatrix} \lambda^{N-1}\\
\lambda^{N-2}\\ \cdot\\ 1\end{pmatrix}\,d\rho(\lambda)
\left(\widetilde\Lambda_N\right)^*=\widetilde\Lambda_N S_0^N
\left(\widetilde\Lambda_N\right)^*,
\end{eqnarray*}
which proves (\ref{CT_S0}).

Using the representation of $B^N$ (\ref{BN_def}) and
(\ref{SP_mes_d}) yields the following formula for entries $b_{ij}$
of $B^N$:
\begin{equation*}
b_{ij}=\int_{-\infty}^\infty
\mathcal{T}_{N-i+1}(\lambda)\left(\mathcal{T}_{N-j+2}(\lambda)+\mathcal{T}_{N-j}(\lambda)\right)\,d\rho(\lambda),\quad
i,j=i,\ldots,N,
\end{equation*}
where we counted that $\mathcal{T}_0(\lambda)=0$. Making use of
(\ref{Chebysh}) leads to:
\begin{equation} \label{Bij}
b_{ij}=\int_{-\infty}^\infty \mathcal{T}_{N-i+1}(\lambda)\lambda
\mathcal{T}_{N-j+1}(\lambda)\,d\rho(\lambda),\quad i,j=i,\ldots,N.
\end{equation}
Then using (\ref{Response_moments_rel}) and (\ref{Bij}) we obtain:
\begin{eqnarray*}
B^N=\int_{-\infty}^\infty J_N\Lambda_N J_NJ_N\lambda\begin{pmatrix} 1\\
\lambda\\ \cdot\\ \lambda^{N-1}\end{pmatrix}\otimes
J_N\Lambda_N J_NJ_N\begin{pmatrix} 1\\
\lambda\\ \cdot\\ \lambda^{N-1}\end{pmatrix}\,d\rho(\lambda)\\
=\int_{-\infty}^\infty \widetilde\Lambda_N\begin{pmatrix} \lambda^{N}\\
\lambda^{N}\\ \cdot\\ \lambda\end{pmatrix}\otimes
\widetilde\Lambda_N \begin{pmatrix} \lambda^{N-1}\\
\lambda^{N-2}\\ \cdot\\ 1\end{pmatrix}\,d\rho(\lambda)\\
=\widetilde\Lambda_N \int_{-\infty}^\infty \begin{pmatrix} \lambda^{N}\\
\lambda^{N-1}\\ \cdot\\ \lambda\end{pmatrix}\otimes \begin{pmatrix} \lambda^{N-1}\\
\lambda^{N-2}\\ \cdot\\ 1\end{pmatrix}\,d\rho(\lambda)
\left(\widetilde\Lambda_N\right)^*=\widetilde\Lambda_N S_1^N
\left(\widetilde\Lambda_N\right)^*,
\end{eqnarray*}
which gives (\ref{BT_S1}). Then (\ref{m_eqn2}) is a consequence of
(\ref{CT_S0}) and (\ref{BT_S1}).
\end{proof}

\subsection{Special cases: Hamburger, Stieltjes and Hausdorff moment problems}
In the previous sections we constructed the special measure
corresponding to operator $A^N$ with Dirichlet boundary condition
at $n=N+1$ and which gives a solution of a truncated moment
problem. Here we formulate several consequences of Theorems
\ref{teor}, \ref{Propos_EL}.

Bearing in mind the relationship between elements of response
vector and moments (\ref{Resp_spectr}) and the formula
(\ref{CT_S0}), we can reformulate Theorem \ref{Th_char} as
\begin{proposition}
\label{Th_char_new} The  set of numbers
$(s_0,s_1,s_2,\ldots,s_{2N-2})$ are moments of a spectral measure
corresponding to the Jacobi operator $A^N$ with Dirichlet boundary
condition at $n=N+1$ if and only if
\begin{equation}
\label{HamCnd} \text{the matrix $S^N_0$ is positive definite.}
\end{equation}

\end{proposition}

The Stieltjes moment problem is characterized by the positivity of
a support of a measure. That means the positivity of a spectrum of
$A^N$. The latter leads to the following
\begin{proposition}
The  set of numbers $(s_0,s_1,s_2,\ldots,s_{2N-1})$ are moments of
a spectral measure, supported on $(0,+\infty)$, corresponding to
Jacobi operator $A^N$ with Dirichlet boundary condition at $n=N+1$
if and only if
\begin{equation}
\label{StielCnd} \text{matrices $S^N_0$ and $S^N_1$ are positive
definite.}
\end{equation}
\end{proposition}

In the Hausdorff moment problem the measure is supported on
$(0,1)$, which leads to the following
\begin{proposition}
The  set of numbers $(s_0,s_1,s_2,\ldots,s_{2N-1})$ are moments of
a spectral measure, supported on $(0,1)$, corresponding to
operator $A^N$ with Dirichlet boundary condition at $n=N+1$, if
and only if the condition
\begin{equation}
\label{Hausd_inq} S^N_0\geqslant S^N_1 > 0
\end{equation}
 holds.
\end{proposition}
\begin{proof}
From (\ref{m_eqn2}) it follows that
\begin{equation}
\lambda_k=\frac{\left(S^N_1g_k,g_k\right)}{\left(S^T_0g_k,g_k\right)},\quad
k=1,\ldots,N.
\end{equation}
Then the restriction $\lambda_k\in (0,1)$ implies
(\ref{Hausd_inq}).
\end{proof}

\begin{remark}
Given an infinite sequence of moments, one can determine whether
or not it is Hamburger or Stieltjes or Hausdorff moment sequence
by verifying condition (\ref{HamCnd}) or (\ref{StielCnd}) or
(\ref{Hausd_inq}) holds for all $N\in \mathbb{N}$.
\end{remark}

\subsection{Recovering Jacobi matrix, nonuniqueness of the solution of the truncated moment problem.}

As we mentioned, given a sequence of moments
$\{s_0,s_1,\ldots,s_{2N-2}\}$ or, equivalently, entries of the
response vector, $\{r_0,r_1,\ldots,r_{2N-2}\}$, it is possible to
recover the Jacobi matrix $A^N$ \cite{MM,MM3}. Introduce the
matrices
\begin{gather*}
\overline C^k=J_kC^kJ_k,\quad k\in \mathbb{N},\\
\overline C^{k-1}_k:=\begin{pmatrix}
\overline c_{1,\,1} & .. & .. &  \overline c_{1,\,k-2} &  \overline c_{1,\,k}\\
.. & .. & ..  &..\\
\cdot & \cdot & \cdot & \cdot  \\
\overline c_{k-1,\,1} &.. &   & \overline c_{k-1,\,k-2}& \overline
c_{k-1,\,k}
\end{pmatrix},\quad k\in \mathbb{N}.
\end{gather*}
that is $\overline C^{k-1}_k$ is constructed from the matrix
$C^{k-1}$ by substituting the last column by $(\overline
c_{1,\,k},\ldots,\overline c_{k-1,\,k})^T$.

\begin{proposition}
The entries of Jacobi matrix can be recovered by
\begin{equation}
\label{AK} a_k=\frac{\left(\det{\overline
C^{k+1}}\right)^{\frac{1}{2}}\left(\det{\overline
C^{k-1}}\right)^{\frac{1}{2}}}{\det{\overline C^{k}}}, \quad
k=1,\ldots, N-1,
\end{equation}
where we set $\det{C^0}=1,$ $\det{C^{-1}=1}$.
\begin{equation}
\label{BK} b_k=-\frac{\det{\overline C^{k}_{k+1}}}{\det{\overline
C^{k}}}+\frac{\det{\overline C^{k-1}_k}}{\det{\overline C^{k-1}}},
\quad k=1,\ldots, N.
\end{equation}
\end{proposition}
Note that (\ref{AK}), (\ref{BK}) are the consequences of the
relation
\begin{equation*}
C^T=\left(W^T\right)^*W^T,\quad \text{or} \quad
\left(\left(W^T\right)^{-1}\right)^*C^T\left(W^T\right)^{-1}=I.
\end{equation*}
and representation (\ref{WT}) of $W^T$.


\begin{remark}
In order to apply the results of Theorems \ref{teor},
\ref{Propos_EL}, i.e. the generalized spectral problem
(\ref{m_eqn2}) to the problem of reconstruction of spectral
measure of $A^N$ one need to know one extra moment, specifically
$s_{2N-1}$ (see the definition of $S^N_1$), than in the method
based on direct calculation of $A^N$ by formulas (\ref{AK}) and
(\ref{BK}).
\end{remark}

Denote by $B(\mathbb{R})$ the space of Borel measures on
$\mathbb{R}$ and by $M_N\subset B(\mathbb{R})$ a subset such that
$d\nu(\lambda)\in M_N$ is a solution of the truncated moment
problem (\ref{Moment_eq}) of the order $N$. We used the BC method
to construct the \emph{special solution} of a truncated moment
problem: for $N\in \mathbb{N}$ the set of moments
$\{s_0,s_1,\ldots,s_{2N-1}\}$ determines the measure
$d\rho^N(\lambda)\in M_N$, where the constructed measure is a
spectral measure of a finite Jacobi operator $A^N$ with the
Dirichlet condition at $n=N+1$. We point out that in our procedure
we do not use the Jacobi matrix, but rather special Hankel
matrices, constructed from moments.

Having constructed the Jacobi matrix $A^N$ from the set
$\{s_0,s_1,\ldots,s_{2N-2}\}$ we can consider the operator
$\widetilde A^N$ given by mixed boundary condition at $n=N+1$:
$\alpha\phi_N+\beta\phi_{N+1}=0$ for some $\alpha,\beta\in
\mathbb{R}$, $|\alpha|+|\beta|>0$, or one can extend the matrix
$A^N$ in any way, keeping it to be Jacobi (it is possible that it
would be necessary to add the boundary condition at infinity).
Then the spectral measure to any of described operators also gives
a solution to Hamburger moment problem.

\begin{remark}
The spectral representation of (\ref{SP_mes_d}) implies that $M_N$
is a convex set in $B(\mathbb{R})$, and obviously
$M_{N_1}\subseteq M_{N_2}$ when $N_1>N_2$. Taking $N$ to infinity
we deduce that the set of solutions $M_\infty$ of the Hamburger
moment problem (\ref{Moment_eq}) either convex, or consists of one
element. The same arguments and spectral representation of $B^T$
(\ref{Bij}) shows that the set of solutions $M_\infty^s$ to
Stieltjes moment problem either convex or consists of one element.
\end{remark}

\section{On the uniqueness of the solution of the Hamburger, Stieltjes and Hausdorff moment problems.}

We remind the reader that the moment problem is called
\emph{determinate} if it has only one solution, otherwise it is
called \emph{indeterminate}.

In this section we use the complex-valued outer and inner spaces
for the dynamical system (\ref{Jacobi_dyn}):
$\mathcal{F}^T=\mathcal{H}^T=\mathbb{C}^T$ with the scalar
products $(f,g)_{\mathcal{F}^T}=\sum_{i=0}^{T-1}f_i\overline g_i$
and $(a,b)_{\mathcal{H}^T}=\sum_{i=1}^Ta_i\overline{b_i}$.

\subsection{Krein equations}
Let $\alpha,\beta\in \mathbb{R}$ and $y(\lambda)$ be a solution to
a Cauchy problem for the following difference equation (we remind
the agreement $a_0=1$):
\begin{equation}
\label{y_special} \left\{
\begin{array}l
a_ky_{k+1}+a_{k-1}y_{k-1}+b_ky_k=\lambda y_k,\\
y_0=\alpha,\,\, y_1=\beta.
\end{array}
\right.
\end{equation}
We set up the \emph{special control problem}: to find a control
$f^T\in \mathcal{F}^T$ that drives the system (\ref{Jacobi_dyn})
to the prescribed state
$\overline{y^T(\lambda)}:=\left(\overline{y_1(\lambda)},\ldots,\overline{y_T(\lambda)}\right)\in
\mathcal{H}^T$ at $t=T$:
\begin{equation}
\label{Control_probl} W^Tf^T=\overline{y^T(\lambda)},\quad
\left(W^Tf^T\right)_k=\overline{y_k(\lambda)},\quad k=1,\ldots,T.
\end{equation}
Note that due to Lemma \ref{teor_control}, this control problem
has a unique solution
$f^T=\left(W^T\right)^{-1}\overline{y^T(\lambda)}$. Let
$\varkappa^T(\lambda)$ be a solution to
\begin{equation}
\label{kappa} \left\{
\begin{array}l
\varkappa^T_{t+1}+\varkappa^T_{t-1}=\lambda\varkappa_t,\quad t=0,\ldots,T,\\
\varkappa^T_{T}=0,\,\, \varkappa^T_{T-1}=1.
\end{array}
\right.
\end{equation}
One can easily see the relation with Chebyshev polynomials
(\ref{Chebysh}):
\begin{equation}
\label{Cheb_rel}
\varkappa^T_t(\lambda)=\mathcal{T}_{T-t}(\lambda),\quad
t=0,1,\ldots,T.
\end{equation}
It is an important fact that the control $f^T$ can be found as a
solution to certain equation:
\begin{theorem}
\label{Theor_Krein} The control
$f^T=\left(W^T\right)^{-1}\overline{y^T(\lambda)}$ solving the
special control problem (\ref{Control_probl}), is a unique
solution to the following Krein-type equation in $\mathcal{F}^T$:
\begin{equation}
\label{C_T_Krein}
C^Tf^T=\beta\overline{\varkappa^T(\lambda)}-\alpha
\left(R^T\right)^*\overline{\varkappa^T(\lambda)}.
\end{equation}
\end{theorem}
\begin{proof}
Let $f^T$ be a solution to (\ref{Control_probl}). We observe that
for any fixed $g\in \mathcal{F}^T$ we have that
\begin{equation}
\label{Kr_1}
u^g_{k,\,T}=\sum_{t=0}^{T-1}\left(u^g_{k,\,t+1}+u^g_{k,\,t-1}-\lambda
u^g_{k,\,t}\right)\varkappa^T_t.
\end{equation}
Indeed, changing the order of a summation in the right hand side
of (\ref{Kr_1}) and counting $u_{k,\,-1}^g=u_{k,\,0}^g=0$  yields
\begin{eqnarray*}
\sum_{t=0}^{T-1}\left(u^g_{k,\,t+1}+u^g_{k,\,t-1}-\lambda
u^g_{k,\,t}\right)\varkappa^T_t=\sum_{t=0}^{T-1}\left(\varkappa^T_{t+1}+\varkappa^T_{t-1}-\lambda\varkappa_t
\right)u^g_{k,\,t}+u^g_{k,\,T}\varkappa^T_{T-1},
\end{eqnarray*}
which gives (\ref{Kr_1}) due to (\ref{kappa}). Using this
observation, we can evaluate
\begin{eqnarray*}
\left(C^Tf^T,g\right)_{\mathcal{F}^T}=\left(u^f(T),u^g(T)\right)_{\mathcal{\mathcal{F}}^T}=\sum_{k=1}^T
u^f_{k,\,T}\overline{u^g_{k\,T}}\\
=\sum_{k=1}^T
\overline{y_k(\lambda)}\overline{u^g_{k,\,T}}=\sum_{k=1}^T
\overline{y_k(\lambda)}\overline{\sum_{t=0}^{T-1}\left(u^g_{k,\,t+1}+u^g_{k,\,t-1}-\lambda
u^g_{k,\,t}\right)\varkappa^T_t}\\
=\sum_{t=0}^{T-1}\overline{\varkappa^T_t(\lambda)}\left(\sum_{k=1}^T
\left(a_k\overline{u^g_{k+1,\,t}}\overline{y_k}+a_{k-1}\overline{u^g_{k-1,\,t}}\overline{y_k} +
b_k\overline{u^g_{k,\,t}}\overline{y_k}-\overline{\lambda} \overline{u^g_{k,\,t}}\overline{y_k}\right)\right)\\
=\sum_{t=0}^{T-1}\overline{\varkappa^T_t(\lambda)}\left(\sum_{k=1}^T
\left(\overline{u^g_{k,\,t}}(a_k\overline{y_{k+1}}+a_{k-1}\overline{y_{k-1}}+b_k\overline{y_k}-
\overline{\lambda y_k} \right)+\overline{u^g_{0,\,t}}\overline{y_1}\right.\\
\left.+a_T{u^g_{T+1,\,t}}\overline{y_T}-\overline{u^g_{1,\,t}}\overline{y_0}-
a_T\overline{u^g_{T,\,t}}\overline{y_{T+1}} \right)
=\sum_{t=0}^{T-1}\overline{\varkappa^T_t(\lambda)}\left(\beta
\overline{g_t}-\alpha\overline{\left(R^Tg\right)_t} \right)\\
=\left(\overline{\varkappa^T(\lambda)}, \left[\beta g-\alpha
\left(R^Tg\right)\right]\right)_{\mathcal{F}^T}=\left(\left[\beta\overline{\varkappa^T(\lambda)}
- \alpha
\left(\left(R^T\right)^*\overline{\varkappa^T(\lambda)}\right)\right],
g\right)_{\mathcal{F}^T}.
\end{eqnarray*}
Which completes the proof due to the arbitrariness of $g$.
\end{proof}



We consider two special solutions to (\ref{y_special}): the first
one $\varphi(\lambda)$ corresponds to the choice $\alpha=0,$
$\beta=1$, the second one, $\xi(\lambda)$, corresponds to Cauchy
data $\alpha=-1$, $\beta=0$.

It is well-known fact \cite{Ahiez,S} that the questions on the
uniqueness of the solution to a moment problem are related to the
index of the operator $A$. Here we provide well-known results on
discrete version of Weyl limit point-circle theory which answers
the question on the index of $A$ that will be subsequently used:
\begin{proposition}
\label{Hamburher_crt} The Jacobi operator $A$ is limit circle at
infinity (has index equal to one) if and only if one of the
following occurs:
\begin{itemize}


\item[1)] $\varphi(x),\xi(x)\in l^2$ for some $x\in \mathbb{R}$,


\item[2)] $\varphi(x),\varphi'(x)\in l^2$ for some $x\in
\mathbb{R}$,

\item[3)] $\xi(x),\xi'(x)\in l^2$ for some $x\in \mathbb{R}$.
\end{itemize}
\end{proposition}

\subsection{Hamburger moment problem}
Let in (\ref{y_special}) $\alpha=0,$ $\beta=1$, then the special
control problem has a form:
\begin{equation}
\label{Contr1} W^Tf^T_{01}(\lambda)=\overline
y^T(\lambda)=\overline{\left(\varphi_1(\lambda),\ldots,\varphi_T(\lambda)\right)}.
\end{equation}
The control $f^T_{01}$ is a unique solution to (see
(\ref{C_T_Krein}) (\ref{Cheb_rel})) the equation
\begin{equation}
\label{Contr2}
C^Tf^T_{01}(\lambda)=\overline{\begin{pmatrix} \mathcal{T}_T(\lambda)\\
\mathcal{T}_{T-1}(\lambda)\\\ldots\\
\mathcal{T}_1(\lambda)\end{pmatrix}}.
\end{equation}
Differentiating (\ref{Contr1}), (\ref{Contr2}) with respect to
$\lambda$, we see that
\begin{equation*}
W^T\left(f^T_{01}(\lambda)\right)'=\overline
{\left(y^T\right)'(\lambda)}=\overline{\left(\varphi_1'(\lambda),\ldots,\varphi_T'(\lambda)\right)},
\end{equation*}
and the control $\left(f^T_{01}(\lambda)\right)'$ is a solution to
\begin{equation*}
C^T\left(f^T(\lambda)\right)'=\overline{\begin{pmatrix} \mathcal{T}_T'(\lambda)\\
\mathcal{T}_{T-1}'(\lambda)\\\ldots\\
\mathcal{T}_1'(\lambda)\end{pmatrix}}.
\end{equation*}
Evaluating the quadratic form (\ref{C_T_def}) we have that
\begin{equation*}
\left(C^T
f^T(\lambda),f^T(\lambda)\right)_{\mathcal{F}^T}=
\left(\overline{\left(\varphi_1(\lambda),\ldots,\varphi_T(\lambda)\right)},\overline{\left(\varphi_1(\lambda),\ldots,\varphi_T(\lambda)\right)}\right)_{\mathcal{H}^T}=\sum_1^T
|\varphi_k(\lambda)|^2.
\end{equation*}
And similarly for the derivatives:
\begin{equation*}
\left(C^T
\left(f^T(\lambda)\right)',\left(f^T(\lambda)\right)'\right)_{\mathcal{F}^T}=\sum_1^T
|\varphi_k'(\lambda)|^2.
\end{equation*}
It is known that
\begin{eqnarray}
\mathcal{T}_{2n-1}(0)=(-1)^{n-1},\quad\mathcal{ T}_{2n}(0)=0,\quad n\geqslant 1,\label{T1}\\
\mathcal{T}_{2n-1}'(0)=0  ,\quad
\mathcal{T}_{2n}'(0)=(-1)^{n-1}n,\quad n\geqslant 1.\label{T2}
\end{eqnarray}
We define the vectors
\begin{equation}
\label{GaOm_def}
\Gamma_T:=\begin{pmatrix} \mathcal{T}_T(0)\\
\mathcal{T}_{T-1}(0)\\\ldots\\
\mathcal{T}_1(0)\end{pmatrix},\quad \Omega_T=\begin{pmatrix} \mathcal{T}_T'(0)\\
\mathcal{T}_{T-1}'(0)\\\ldots\\
\mathcal{T}_1'(0)\end{pmatrix}.
\end{equation}
Using the above arguments we can state that
\begin{eqnarray}
\sum_1^T |\varphi_k(0)|^2=\left(\left(C^T\right)^{-1}\Gamma_T,\Gamma_T\right)_{\mathcal{F}^T},\label{Sum_phi}\\
\sum_1^T
|\varphi_k'(0)|^2=\left(\left(C^T\right)^{-1}\Delta_T,\Delta_T\right)_{\mathcal{F}^T}\notag.
\end{eqnarray}
Now we can use 2) from Proposition \ref{Hamburher_crt}, and
formulate the following
\begin{proposition}
The Hamburger moment problem is indeterminate if and only if
\begin{equation*}
\lim_{T\to\infty}\left(\left(C^T\right)^{-1}\Gamma_T,\Gamma_T\right)_{\mathcal{F}^T}<+\infty,\quad
\lim_{T\to\infty}\left(\left(C^T\right)^{-1}\Delta_T,\Delta_T\right)_{\mathcal{F}^T}<+\infty,
\end{equation*}
where $\Gamma_T$ and $\Omega_T$ are defined by (\ref{GaOm_def}),
(\ref{T1}), (\ref{T2}).
\end{proposition}

\subsection{Stieltjes moment problem}

It is known \cite{Ahiez} that the Jacobi matrix in this case
admits the special structure:
\begin{eqnarray}
b_i=\frac{1}{m_i}\left(\frac{1}{l_{i-1}}+\frac{1}{l_i}\right),\quad i=2,3,\ldots,\quad b_1=\frac{1}{m_1 l_1},\label{B_coeff}\\
a_i=\frac{1}{l_{i}\sqrt{m_i m_{i+1}}},\quad i=1,2,\ldots\notag
\end{eqnarray}
where $l_i$, $m_i$ are positive and are interpreted as lengths of
intervals and masses at the points $x_j$. The string is defined by
the density $dM=\sum_{k=1}^\infty m_k\delta(x-x_k)$, where
$0=x_0<x_1<x_2<\ldots<x_{N-1}<\ldots$, $l_i=x_i-x_{i-1}$,
$i=1,\ldots$. The inverse dynamic problem for the dynamical system
corresponding to a finite Krein-Stieltjes string was studied in
\cite{MM4}. It is straightforward to check (see also \cite{Ahiez})
that the following relations hold:
\begin{eqnarray}
\varphi_n(0)=(-1)^n\sqrt{m_{n}},\quad n=1,2,\ldots,\label{ML1}\\
\xi_n(0)=(-1)^{n-1}\left(\sum_{j=1}^{n-1}
l_j\right)\sqrt{m_{n}}\quad n=1,2,\ldots.\label{ML2}
\end{eqnarray}
We define the mass and length of a segment of a string:
\begin{equation*}
M_K=\sum_{k=1}^K m_k;\quad L_K=\sum_{k=1}^K l_k,
\end{equation*}
when $K=\infty$ above expressions correspond to the mass and the
length of the whole string. Then formulas (\ref{ML1}), (\ref{ML2})
imply that
\begin{equation*}
M_K=\sum_{n=1}^K\varphi_n^2(0), \quad L_K=
-\frac{\xi_{K+1}(0)}{\varphi_{K+1}(0)}.
\end{equation*}
The immediate consequence of 1) in Proposition \ref{Hamburher_crt}
and formulas (\ref{ML1}), (\ref{ML2}) is the following
\begin{proposition}
The Stieltjes moment problem is indeterminate if and only if both
length and mass of a string is finite:
$M_\infty,\,L_\infty<+\infty$.
\end{proposition}
For the mass of the segment of a string or of the whole string we
have an expression (\ref{Sum_phi}). Now we obtain similar formula
for the length. Denote by $h^T\in \mathcal{F}^T$ the (unique)
control which drives the system (\ref{Jacobi_dyn}) to the special
state
\begin{equation*}
W^Th^T=(0,\ldots,0,1)\in \mathcal{H}^T.
\end{equation*}
For arbitrary $q\in \mathcal{F}^T$ we have (see the representation
(\ref{WT})) that
\begin{equation*}
\left(C^Th^T,q\right)_{\mathcal{F}^T}=\left(W^Th^T,W^Tq\right)_{\mathcal{H}^T}=\prod_{i=0}^{T-1}
a_iq_0.
\end{equation*}
The above relation implies that $h^T$ can be found as a unique
solution to Krein-type equation:
\begin{equation}
\label{CT_h}
C^Th^T=\begin{pmatrix} \prod_{i=0}^{T-1}a_i \\0\\
\ldots\\0\end{pmatrix}.
\end{equation}
Taking a control $f^T\in \mathcal{F}^T$ such that
$W^Tf^T=\left(\varphi_1(0),\ldots,\varphi_T(0)\right)$ we have
(see Theorem \ref{Theor_Krein} and (\ref{CT_h}) that
\begin{eqnarray*}
\varphi_T(0)=\left(\begin{pmatrix}0\\0\\\cdot\\1\end{pmatrix},\begin{pmatrix}\varphi_1(0)\\\varphi_2(0)\\\cdot\\
\varphi_T(0)\end{pmatrix}\right)_{\mathcal{H}^T}
=\left(W^Th^T,W^Tf^T\right)_{\mathcal{H}^T}\\
=\left(C^Th^T,f^T\right)_{\mathcal{F}^T}=
\left(\begin{pmatrix} \prod_{i=0}^{T-1}a_i \\0\\
\ldots\\0\end{pmatrix},\left(C^T\right)^{-1}{\begin{pmatrix} \mathcal{T}_T(0)\\
\mathcal{T}_{T-1}(0)\\\ldots\\
\mathcal{T}_1(0)\end{pmatrix}}\right)_{\mathcal{F}^T}.
\end{eqnarray*}
Similarly, denote by $g^T$ the control for which
$W^Tg^T=\left(\xi_1(0),\ldots,\xi_T(0)\right)$. Then, using
equations from Theorem \ref{Theor_Krein} and (\ref{CT_h}) we have
that
\begin{eqnarray*}
\xi_T(0)=\left(\begin{pmatrix}0\\0\\\cdot\\1\end{pmatrix},\begin{pmatrix}\xi_1(0)\\\xi_2(0)\\\cdot\\
\xi_T(0)\end{pmatrix}\right)_{\mathcal{H}^T}=\left(W^Th^T,W^Tg^T\right)_{\mathcal{H}^T}\\
=\left(C^Th^T,g^T\right)_{\mathcal{F}^T}=-\left(\begin{pmatrix} \prod_{i=0}^{T-1}a_i \\0\\
\ldots\\0\end{pmatrix},\left(C^T\right)^{-1}\left(R^T\right)^* {\begin{pmatrix}\mathcal{T}_T(0)\\
\mathcal{T}_{T-1}(0)\\\ldots\\
\mathcal{T}_1(0)\end{pmatrix}}\right)_{\mathcal{F}^T}.
\end{eqnarray*}
The above arguments leads to the following expression for the
length of the segment of the string:
\begin{equation*}
L_K=-\frac{\xi_{K+1}(0)}{\varphi_{K+1}(0)}=\frac{\left(\left(C^{K+1}\right)^{-1}\left(R^{K+1}\right)^*\Gamma_{K+1},e_1\right)}
{\left(\left(C^{K+1}\right)^{-1}\Gamma_{K+1},e_1\right)},
\end{equation*}
where we denoted $e_1=(1,0,\ldots,0)\in \mathcal{F}^T$. The above
arguments leads to the following statement:
\begin{proposition}
\label{PropStieltjes} The Stieltjes moment problem is
indeterminate if and only if the following relations hold:
\begin{eqnarray*}
M_\infty=\lim_{T\to\infty}\left(\left(C^T\right)^{-1}\Gamma_T,\Gamma_T\right)_{\mathcal{F}^T}<+\infty,\\
L_\infty=\lim_{K\to\infty}\frac{\left(\left(C^{K}\right)^{-1}\left(R^{K}\right)^*\Gamma_{K},e_1\right)}
{\left(\left(C^{K}\right)^{-1}\Gamma_{K},e_1\right)}<+\infty.
\end{eqnarray*}
\end{proposition}

\subsection{Hausdorff moment problem}

It is a special case of a Stieltjes moment problem. The necessary
and sufficient conditions for a set of numbers to be a moments of
a measure supported on $(0,1)$ are obtained in \cite{H}, see also
\cite{T}. They are equivalent to inequality (\ref{Hausd_inq})
holds for all $T\in \mathbb{N}$ since we get the limit measure as
a limit of measures supported on $(0,1)$.
\begin{proposition}
If the Halkel matrices $S_0^T$, $S_1^T$ satisfy (\ref{Hausd_inq})
for all $T\geqslant 2$, then there exists only one measure
supported on $(0,1)$ which satisfies (\ref{Moment_eq}). In other
words, the Hausdorff moment problem is determinate.
\end{proposition}
\begin{proof}
Let us assume that the opposite is true and the Hausdorff moment
problem is indeterminate, in this case by Proposition
\ref{PropStieltjes} the length and the mass of string determined
by the matrix $A$ constructed from moments, should be finite. Then
for any fixed $T\in \mathbb{N}$ we have on the one hand that
\begin{equation}
\label{TrA1} \operatorname{tr}A^T=\sum_{n=1}^T\lambda_n^T.
\end{equation}
On the other hand (see (\ref{B_coeff})),
\begin{equation}
\label{TrA2}
\operatorname{tr}A^T=\frac{1}{m_1l_1}+\sum_{n=2}^T\frac{1}{m_n}\left(\frac{1}{l_{n-1}}+\frac{1}{l_n}\right).
\end{equation}
Since $l_i,\,m_i$ are positive and by our assumption
$\sum_{i=1}^\infty m_i<+\infty,\,\, \sum_{i=1}^\infty l_i
<+\infty$, it immediately follows from (\ref{TrA2}) that
$\operatorname{tr}A^T>T^2$ for sufficiently large $T$, and thus
from (\ref{TrA1}) we can see that for such $T$ the eigenvalues
$\lambda^T_k$ cannot be bounded by one. Which gives a
contradiction.
\end{proof}

Let $s\in \mathbb{R}^\mathbb{N}$ be a sequence,
$s=\left(s_0,s_1,\ldots\right)$. One defines the difference
operator $\Delta: \mathbb{R}^\mathbb{N}\mapsto
\mathbb{R}^\mathbb{N}$ by the rule
\begin{equation}
\label{def_del}
(\Delta s)_n=s_{n+1}-s_n, \quad n=0,1,\dots.
\end{equation}
Hausdorff in \cite{H1,H2} proved the following
\begin{theorem}
A sequence $s=\left(s_0,s_1,\ldots\right)\in
\mathbb{R}^\mathbb{N}$ is a moment sequence of a measure supported
on $(0,1)$ if and only if it is completely monotonic, i.e., its
difference sequences satisfy the equalities
\begin{equation}
\label{H_cond}
(-1)^k(\Delta^k s)_n\geqslant 0,\quad \text{for all $k,n\geqslant
0$}.
\end{equation}
\end{theorem}
We will show that the Hausdorff condition (\ref{H_cond}) is a
consequence of condition (\ref{Hausd_inq}) holding far all
$N\in\mathbb{N}$. On the other hand, the property (\ref{H_cond})
for finite $k,n$ does not imply (\ref{Hausd_inq}). The following
two propositions confirm that.

\begin{proposition}
The condition  (\ref{Hausd_inq}) implies the inequality
(\ref{H_cond}) holds for $k,\,n$ such that $k+n\leqslant 2N-1$.
\end{proposition}
\begin{proof}
According to definition (\ref{def_del})
$$
\left(\Delta^2 s\right)_n = \left(\Delta\left(\Delta
s\right)\right)_n = \left(\Delta s\right)_{n+1}-\left(\Delta
s\right)_{n}=s_{n+2}-2s_{n-1}+s_n,
$$
continuing calculations yields
\begin{equation}\label{dsn}
\left(\Delta^k s\right)_n=\sum_{i=0}^{k}(-1)^i C^i_k s_{n+k-i}.
\end{equation}

For a given sequence $g=\left(g_0,g_1,\ldots\right)\in
\mathbb{R}^\mathbb{N}$ by $\left(g_{i+j-2}\right)_{1\leq i,j\leq
n+1}$ we denote the Hankel matrix, and the Hankel transform get
map $g\in \mathbb{R}^\mathbb{N}$ to the sequence
$$
h_{n}=\det(g_{i+j-2})_{1\leq i,j\leq n+1},\quad n=0,1,\dots
$$
The  binomial transform get map a given sequence
$g=\left(g_0,g_1,\ldots\right)\in \mathbb{R}^\mathbb{N}$ to the
sequence
\begin{equation}\label{bin}
c_{n}=\sum _{k=0}^{n}C_n^k g_{k},\quad n=0,1,\dots
\end{equation}
It is known \cite{P} that the Hankel transform is invariant under
the binomial transform, that is:
\begin{equation}
\label{hankel_prop}
h_n =\det(g_{i+j-2})_{1\leq i,j\leq
n+1}=\det(c_{i+j-2})_{1\leq i,j\leq n+1}.
\end{equation}
Note that (\ref{hankel_prop}) remains valid if one replaces
binomial transform (\ref{bin}) by the signed binomial transform
introduced by the rule: for $g=\left(g_0,g_1,\ldots\right)\in
\mathbb{R}^\mathbb{N}$ one has
\begin{equation*}
\widetilde c_{n}=\sum _{k=0}^{n}(-1)^kC_n^k
g_{k},\quad n=0,1,\dots
\end{equation*}
This fact and (\ref{dsn}) imply that for Hankel matrices
$J_NS^N_0J_N=\left(s_{i+j-2}\right)_{1\leq i,j\leq N}$,
$J_NS^N_1J_N=\left(s_{i+j-1}\right)_{1\leq i,j\leq N}$ and their
difference the following relations hold:
\begin{gather}
\label{hankel_con}
\det(s_{i+j+m-2})_{1\leq i,j\leq n+1}=\det\left(\left(\Delta^{i+j-2} s\right)_m\right)_{1\leq i,j\leq n+1},\quad m=0,1.\\
\det(s_{i+j-2}-s_{i+j-1})_{1\leq i,j\leq
n+1}=\det\left(\left(\Delta^{i+j-2} s\right)_0
-\left(\Delta^{i+j-2} s\right)_1\right)_{1\leq i,j\leq n+1},\notag
\end{gather}
where $n=0,1,\ldots,N-1$. Then the Sylvester criterion of
positivity of a matrix and the condition (\ref{Hausd_inq}) imply
that
\begin{equation*}
\left(\left(\Delta^{i+j-2} s\right)_0\right)_{1\leq i,j\leq
N}\geqslant \left(\left(\Delta^{i+j-2} s\right)_1\right)_{1\leq
i,j\leq N}>0.
\end{equation*}
The above inequalities imply that diagonal elements of matrices
satisfy the same inequalities
\begin{equation}\label{i1}
\left(\Delta^{2i} s\right)_0 \geqslant \left(\Delta^{2i} s\right)_1>0,\quad i=0,1,\dots,N-1.
\end{equation}
Using the definition of $\Delta$ (\ref{def_del}) from (\ref{i1})
we derive that
\begin{equation}\label{i2}
\left(\Delta^{2i+1} s\right)_0 \leqslant 0,\quad i=0,1,\dots,N-1.
\end{equation}
Note that inequalities in (\ref{i1}),(\ref{i2}) are only part of
what we need to show in (\ref{H_cond}). To prove the other part we
note that condition (\ref{Hausd_inq}) implies that
\begin{equation}\label{Hausd_inq2}
 S^{N-m}_{2m}\geqslant S^{N-m}_{1+2m} > 0, \quad  m=0,1,\dots,N-1.
\end{equation}
Like (\ref{i1}) was a consequence of (\ref{Hausd_inq}), the
inequality
\begin{equation}\label{ii1} \left(\Delta^{2i}
s\right)_{2m} \geqslant \left(\Delta^{2i} s\right)_{2m+1}>0,\quad
i=0,1,\dots,N-m-1
\end{equation}
follows from (\ref{Hausd_inq2}). From (\ref{ii1}) one obtains that
\begin{equation}\label{ii2}
\left(\Delta^{2i+1} s\right)_{2m} \leqslant 0,\quad i=0,1,\dots,N-m-1.
\end{equation}
The inequalities (\ref{ii1}), (\ref{ii2}) is exactly what we need
to show in (\ref{H_cond}).
\end{proof}

\begin{proposition}
The inequality (\ref{H_cond}) holding for all $k,n: \ k+n\leqslant
2N-1 $ does not implies (\ref{Hausd_inq}).
\end{proposition}
\begin{proof}
Consider the point mass measure concentrated at $\lambda=1/2$:
i.e. $d\rho(\lambda)=\delta(\lambda-1/2)$. In this case all
moments (\ref{Moment_eq}) are given by $s_k=(1/2)^k$. Using
(\ref{def_del}) we see that $\left(\Delta^k s\right)_n=(-1)^k
(1/2)^{k+n}$, which agrees with (\ref{H_cond}).

Now we construct counterexample which works even for $N=2$. To do
so we slightly change the moment $s_2$: we take $s_0=1$,
$s_1=1/2$, $s_2=1/4+\varepsilon$, $s_3=1/8$. If $\varepsilon$ is
small enough then (\ref{H_cond}) remains valid, but
(\ref{Hausd_inq}) fails: indeed, $\det(s_{i+j-2})_{1\leq i,j\leq
2}=\varepsilon<0$ if $\varepsilon<0$. Therefore (\ref{Hausd_inq})
doesn't follow from (\ref{H_cond}).
\end{proof}

Note that the above counterexample does not work in the case of
infinite matrices: in this case the restriction on indices
$k+n\leqslant 2N-1$ disappears and one should consider all $k,n$.
Then according to (\ref{dsn})
$$
\left(\Delta^k s\right)_n=\sum_{i=0}^{k}(-1)^i C^i_k s_{n+k-i} =
(-1)^k(1/2)^{k+n} +  (-1)^{n+k-2}\varepsilon C^k_{n+k-2},
$$
and we see that the second term in the right hand side of the
above expression dominates if $k,n$ are large enough. Therefore
for such a moment sequence  the condition (\ref{H_cond}) does not
hold.

\subsection*{Acknowledgments}
The research of Victor Mikhaylov was supported in part by RFBR
17-01-00529, RFBR 18-01-00269. Alexandr Mikhaylov was supported by
RFBR 17-01-00099 and RFBR 18-01-00269.


\begin{thebibliography}{99}

\bibitem{Ahiez}
{N.I. Akhiezer.} The classical moment problem and some related
questions in analysis. Edinburgh: Oliver and Boyd, 1965.


\bibitem{AM}
{S.A. Avdonin, V.S. Mikhaylov.} \textit{The boundary control
approach to inverse spectral theory,} Inverse Problems {\bf 26},
2010, no. 4, 045009, 19 pp.

\bibitem{AMM} {S.A.Avdonin, A.S. Mikhaylov, V.S. Mikhaylov.} \textit{On some
applications of the Boundary Control method to spectral estimation
and inverse problems.} Nanosystems: Physics, Chemistry,
Mathematics, 6, no. 1, 63--78, 2015.

\bibitem{A}
{F. V. Atkinson} \textit{Discrete and continuous boundary
problems}, Acad. Press, 1964.


\bibitem{B07}
{M.I. Belishev}, \textit{Recent progress in the boundary control
method}, Inverse Problems, {\bf 23}, (2007).


\bibitem{B2001}
{M.I. Belishev}. \textit{On relations between spectral and
dynamical inverse data,} J. Inverse Ill-Posed Probl. 2001. V.~9,
iss.~6. P.~547--565.

\bibitem{B17}
{M.I. Belishev.} \textit{Boundary control and tomography of
Riemannian manifolds (the BC-method),} Uspekhi Matem. Nauk. 2017.
V.~72, iss.~4. P.~3--66, (in Russian).

\bibitem{BM}
{M.I.Belishev, V.S.Mikhaylov}. \textit{Unified approach to
classical equations of inverse problem theory.} {Journal of
Inverse and Ill-Posed Problems}, 20 (2012), no 4, 461--488.

\bibitem{BL}
{A. S. Blagoveschenskii.} \textit{On a local approach to the
solution of the dynamical inverse problem for an inhomogeneous
string.} Trudy MIAN, {\bf 115}, 28--38, 1971 (in Russian).


\bibitem{SG}
{F. Gesztesy, B. Simon}, \textit{$m-$functions and inverse
spectral analisys for dinite and semi-infinite Jacobi matrices},
J. d'Analyse Math. 73 (1997), 267-297

\bibitem{H1}
{F. Hausdorff},  \textit{Summationsmethoden und Momentfolgen. I.},
Mathematische Zeitschrift 9, 74--109, 1921.

\bibitem{H2}
{F. Hausdorff}, \textit{Summationsmethoden und Momentfolgen. II.},
Mathematische Zeitschrift 9, 280--299, 1921.

\bibitem{MM}
{A. S. Mikhaylov, V. S. Mikhaylov.} \textit{Dynamical inverse
problem for the discrete Schr\"odinger operator.} Nanosystems:
Physics, Chemistry, Mathematics., 7, (5), 842-854, 2016.


\bibitem{MM1}
{A. S. Mikhaylov, V. S. Mikhaylov.} \textit{Inverse dynamic
problems for canonical systems  and de Branges spaces.}
Nanosystems: Physics, Chemistry, Mathematics, 9, no. 2, 215-224,
2018.

\bibitem{MM2}
{A. S. Mikhaylov, V. S. Mikhaylov.} \textit{Boundary Control
method and de Branges spaces. Schr\"odinger operator, Dirac
system, discrete Schr\"odinger operator.} Journal of Mathematical
Analysis and Applications, 460, no. 2, 927-953, 2018.

\bibitem{MM3}
{A. S. Mikhaylov, V. S. Mikhaylov.} \textit{Dynamic inverse
problem for Jacobi matrices,} Inverse Problems and Imaging, 13,
no. 3, 2019.

\bibitem{MM4}
{A.S. Mikhaylov, V.S. Mikhaylov,} \textit{Inverse dynamic problem
for a Krein-Stieltjes string.} Applied Mathematics Letters
https://doi.org/10.1016/j.aml.2019.05.002


\bibitem{MMS}
{A. S. Mikhaylov, V. S. Mikhaylov, S.A. Simonov.} \textit{On the
relationship between Weyl functions of Jacobi matrices and
response vectors for special dynamical systems with discrete
time,} Mathematical Methods in the Applied Sciences, 2018.

\bibitem{P}
{J.R. Partington},\textit{An introduction to Hankel operators. LMS
Student Texts.}, Cambridge University Press. ISBN 0-521-36791-3,
1988

\bibitem{S}
{B. Simon}, \textit{The classical moment problem as a self-adjoint
finite difference operator.},  Advances in Math., 137, 1998,
82-203.

\bibitem{T}
{J. A. Shohat, J. D. Tamarkin}, \textit{The Problem of Moments,}
New York: American mathematical society, 1943.

\end{thebibliography}
\end{document}